\documentclass{amsart}

\usepackage[all]{xy}
\usepackage{amssymb,latexsym}
\usepackage{mltex}
\usepackage{amsmath}
\usepackage[colorlinks=true,linkcolor=blue,citecolor=blue]{hyperref}
 \usepackage{lscape}
\usepackage{enumerate}
\usepackage{amsthm}
\usepackage{hyperref}
\usepackage{multicol}

\newtheorem{thm}{Theorem}
\newtheorem{cor}{Corollary}
\newtheorem{lem}{Lemma}[section]
\newtheorem{prop}[lem]{Proposition}
\newtheorem{rem}{Remark}

\newtheorem*{claim}{Claim}

\newcommand{\spin}{\mathrm{Spin}}

\newcommand{\C}{\mathbb{C}}

\newcommand{\HH}{\mathbb{H}}

\newcommand{\R}{\mathbb{R}}

\newcommand{\Z}{\mathbb{Z}}

\newcommand{\beqt}{\begin{equation}}  \newcommand{\eeqt}{\end{equation}}
\newcommand{\bal}{\begin{align}}      \newcommand{\eal}{\end{align}}
\newcommand{\ba}{\begin{array}}      \newcommand{\ea}{\end{array}}
\newcommand{\bc}{\begin{center}}     \newcommand{\ec}{\end{center}}
\newcommand{\be}{\begin{enumerate}}  \newcommand{\ee}{\end{enumerate}}
\newcommand{\beq}{\begin{eqnarray}}  \newcommand{\eeq}{\end{eqnarray}}
\newcommand{\beQ}{\begin{eqnarray*}} \newcommand{\eeQ}{\end{eqnarray*}}
\newcommand{\bi}{\begin{itemize}}    \newcommand{\ei}{\end{itemize}}
\newcommand{\bt}{\begin{tabular}}    \newcommand{\et}{\end{tabular}}

\begin{document}
\title{Spinorial Representation of Submanifolds in Riemannian space forms}
\author{Pierre Bayard, Marie-Am\'elie Lawn and  Julien Roth}

\maketitle

\begin{abstract}
In this paper we give a spinorial representation of submanifolds of any dimension and codimension into Riemannian space forms in terms of the existence of so called generalized Killing spinors. We then discuss several applications, among them a new and concise proof of the fundamental theorem of submanifold theory. We also recover results of T. Friedrich, B. Morel and the authors in dimension 2 and 3. \end{abstract}
{\it keywords:} Spin geometry, Isometric Immersions, Weierstrass Representation.

\date{}
\maketitle\pagenumbering{arabic}
\section{Introduction}
One of the fundamental problems in submanifold theory deals with the existence of isometric
immersions from a Riemannian manifold $M^n$ into another fixed Riemannian manifold $N^{n+p}$. If the ambient manifold is a space form $\mathbb{R}^{n+p}$, $\mathbb{S}^{n+p}$ or $\mathbb{H}^{n+p}$, the fundamental theorem of submanifold theory states that the Gau\ss, Ricci and Codazzi equations, also called structure equations, are necessary and sufficient conditions.

 In the case of surfaces, another approach is given by the study of Weierstrass representations. Historically, these representations are describing a conformal minimal immersion of a Riemann surface $M$ into the three-dimensional Euclidean space $\mathbb{R}^3$. Precisely, given a pair $(h, g)$ consisting of a holomorphic and a meromorphic function, the formula
$$f(x, y) = \Re e\int\left((1-g^2(z))h(z), (1+g^2(z))h(z), 2g(z)h(z)\right)dz,$$
with $z =x+iy$ some complex coordinate, gives a local parametrization of a
minimal surface in Euclidean three-space. Conversely every minimal surface
can be parametrized in this way with respect to isothermal coordinates. However, relaxing the condition of holomorphicity on the pair $(h,g)$, this representation is much more general and can actually describe all surfaces in $\mathbb{R}^3$ as shown in \cite{Ke}. 

At the end of the 1990s, following an idea of U. Abresch and D. Sullivan, R. Kusner and N.
Schmidt reformulated this approach in a more concise and simpler way in terms of spinor fields (see \cite{KS}). These so called spinorial Weierstrass representations were studied extensively by B.G. Konopelchenko, I. Taimanov and many others (see \cite{Ko,KL,Ta,Ta2} and the references there).

However these formulae were given in local coordinates and remained purely computational until Friedrich gave in \cite{Fr} an elegant and geometrically invariant description using spinor bundles. We point out that the equivalence between the two approaches was recently showed in \cite{RR}. The main idea is to use the identification between the ambient spinor bundle restricted to the surface and the spinor bundle of the surface.  Note that the condition to be a spin manifold is not restrictive here since any oriented surface is also spin. More generally, the restriction $\varphi$ of a parallel spinor field on $\mathbb{R}^{n+1}$ to an 
oriented Riemannian hypersurface $M^n$ is a solution of a \textit{generalized Killing
equation}
$$\nabla^{\Sigma M}_X\varphi=A(X)\cdot\varphi,$$ 
where  $\nabla^{\Sigma M}$ and $\cdot$ are respectively the  spin connection and the Clifford multiplication on $M$, and $A$ is the shape operator of the immersion. Conversely, Friedrich showed that in the case where $M$ is a surface, if there exists a particular spinor field $\varphi$ satisfying the generalized Killing equation, where $A$ is an arbitrary field of symmetric endomorphisms  of the tangent bundle, then there exists an isometric immersion of $M$ into $\mathbb{R}^3$ with shape operator $A$. Moreover $\varphi$ is the restriction to $M$ of a parallel spinor of $\mathbb{R}^3$. The proof consists in showing that $A$ indeed satisfies the structure equations. This result was generalized to surfaces into other  three-dimensional ambient spaces \cite{Mo,NR,Ro}, to three-dimensional manifolds into four-dimensional space forms \cite{LR1,NR} and also to the two-dimensional pseudo-Riemannian setting \cite{LR2}. However the question whether in general a manifold of arbitrary dimension carrying a generalized Killing spinor can be immersed isometrically into some Euclidean space remained until now unanswered. Some of the few achievement in this direction were obtained in \cite{AMM} for real analytic manifolds and in \cite{BGM,Na} when $A$ is a Codazzi tensor, showing the existence of an immersion into a Ricci flat manifold admitting a parallel spinor which restricts to $\varphi$. 

Similarly, in higher codimension, very little is known. In \cite{BLR}, we extended the approach to the case of surfaces in four-dimensional space forms. The key point was to use the remark due to B\"ar \cite{Ba} that an ambient spinor restricted to an immersed submanifold $M$ can be identified with a section of the spinor bundle of the submanifold twisted with the spin bundle of the normal bundle. This was then extended to the pseudo-Riemannian setting in \cite{Bay,BP}.

Following the same idea, we use in this paper a particular twisted spin bundle over a spin manifold of arbitrary dimension to give a geometrically invariant spinorial representation of submanifolds of Euclidean spaces in any codimension. Note that our proof does not use the structure equations but merely the existence of a generalized Killing spinor on the manifold. We later show that one indeed recovers the previously mentioned result of Friedrich \cite{Fr} in the case of surfaces in $\R^3,$ as well as the one of Lawn-Roth \cite{LR1} for three-dimensional hypersurfaces and of Bayard-Lawn-Roth \cite{BLR} for surfaces in $\R^4$ (section 7). It is worth pointing out that the study of generalized Killing spinors has revealed very interesting applications. Moroianu and Semmelmann were for instance able to construct new examples of Lagrangian submanifolds of the nearly K\"ahler $\mathbb{S}^3\times\mathbb{S}^3$ using the existence of such spinors on the sphere $\mathbb{S}^3$ \cite{MS}. Moreover it is well-known that there is a close relationship to G-structures: for instance a generalized Killing spinor defines a cocalibrated G2-structure on the manifold in dimension 7 and a half-flat SU(3)-structure in dimension 6 (see for example \cite{CS}). However the existence of such spinors is a non-trivial problem: our construction is therefore of particular interest.

 Besides the above-mentioned, we discuss several other applications. A notable achievement is a new and concise proof of the fundamental theorem of submanifold theory. In the special case of surfaces, we show that our approach is equivalent to the spinorial Weierstrass representations, i.e., we obtain explicit formulas in terms of functions involving the components of the spinor field which are holomorphic if the surface is minimal.  
Our result can thus be seing as a generalization of most of the concrete Weierstrass representation formulae existing in the literature: it provides a general framework to understand formulae appearing in a variety of contexts. Moreover, since the basic ideas and constructions behind our representation are fairly simple, we hope that our result will be useful to obtain new concrete Weierstrass representation formulae, once some geometric context is specified: this is especially interesting for surfaces, in low-dimensional pseudo-Riemannian space forms, under some curvature assumptions.

Finally, in the last section, we extend our result to submanifolds immersed into the other space forms $S^n$ and $\HH^n,$ and recover the results of Morel \cite{Mo} if $n=3.$ 

\section{Preliminaries}\label{sec preliminaries}
\subsection{The spin representation}
Let us denote by $Cl_n$ the real Clifford algebra on $\R^n$ with its standard scalar product. We consider the representation
\begin{eqnarray*}
\rho:\hspace{.5cm} Cl_n&\rightarrow& End(Cl_n)\\
a&\mapsto& \xi\mapsto a\xi
\end{eqnarray*}
and its restriction to the group $Spin(n)$
\begin{eqnarray*}
\rho_{|Spin(n)}:\hspace{.5cm} Spin(n)&\rightarrow& GL(Cl_n)\\
a&\mapsto& \xi\mapsto a\xi.
\end{eqnarray*}
Note that this is not the adjoint representation of the spin group on the Clifford algebra, but rather the representation given by left multiplication.\\ Moreover we want to point out that we are not taking as usual the restriction of an irreducible representation of the Clifford algebra to the spin group, but that we consider instead the restriction of the entire real Clifford algebra. This real representation splits into a sum of  $2^k$ copies of spinor spaces of dimension $2^{n-k}$, where the number $k$ depends on the dimension $n$ and can be computed using the Radon-Hurwitz numbers (we refer to \cite{Lu} for further details). \\
If $p+q=n,$ we have a natural map
$$Spin(p)\times Spin(q)\rightarrow Spin(n)$$
associated to the splitting $\R^n=\R^p\oplus\R^q$ and to the corresponding isomorphism of Clifford algebras
$$Cl_n=Cl_p\hat{\otimes} Cl_q,$$
where $\hat{\otimes}$ denotes the $\Z_2-$graded tensor product. \\
We get thus the following representation, still denoted by $\rho,$
\begin{eqnarray}
\rho:\hspace{.5cm} Spin(p)\times Spin(q)&\rightarrow& GL(Cl_n)\label{rep spin p spin q}\\
a&\mapsto& \xi\mapsto a\xi.\nonumber
\end{eqnarray}
\subsection{The twisted spinor bundle $\Sigma$}
We consider $M$ a $p$-dimensional Riemannian manifold, $E\rightarrow M$ a bundle of rank $q,$ with a fibre metric and a compatible connection. We assume that $E$ and $TM$ are oriented and spin, with given spin structures
$$\tilde{Q}_M\stackrel{2:1}{\rightarrow} Q_M\hspace{.5cm}\mbox{and}\hspace{.5cm}\tilde{Q}_E\stackrel{2:1}{\rightarrow} Q_E$$
where $Q_M$ and $Q_E$ are the bundles of positively oriented orthonormal frames of $TM$ and $E,$ and we set
$$\tilde{Q}:=\tilde{Q}_M\times_M \tilde{Q}_E;$$
this is a $Spin(p)\times Spin(q)$-principal bundle. We define the associated bundle 
$$\Sigma:=\tilde{Q}\times_{\rho} Cl_n$$
and its restriction
\begin{equation}\label{def usigma}
U\Sigma:=\tilde{Q}\times_{\rho} Spin(n)\hspace{.5cm}\subset\ \Sigma
\end{equation}
where $\rho$ is the representation (\ref{rep spin p spin q}) given by left multiplication. We remark that if we used the adjoint representation instead, we would just get the Clifford algebra bundle. Again we point out that our spinor bundle $\Sigma$ is a real vector bundle with fiber the entire Clifford algebra and not, as usual, an irreducible complex Clifford module.\\
The vector bundle $\Sigma$ is equipped with the covariant derivative $\nabla$ naturally associated to the spinorial connections on $\tilde{Q}_M$ and $\tilde{Q}_E.$ 
\begin{rem}\label{id sigma tensor product} The bundle $\Sigma$ is a spinor bundle on $TM$ twisted by a spinor bundle on $E:$ indeed, let us consider the representations
$$\rho_1:\ Spin(p)\rightarrow GL(Cl_p)\hspace{1cm}\mbox{and}\hspace{1cm}\rho_2:\ Spin(q)\rightarrow GL(Cl_q)$$
given by left multiplication, and the associated bundles
$$\Sigma_1:=\tilde{Q}_M\times_{\rho_1} Cl_p\hspace{1cm}\mbox{and}\hspace{1cm}\Sigma_2:=\tilde{Q}_E\times_{\rho_2} Cl_q$$
equipped with their natural connections $\nabla^1$ and $\nabla^2;$ then
$$\Sigma_1\otimes\Sigma_2\ \simeq\ \Sigma$$
and
$$\nabla^1\otimes id_{\Sigma_2}\ \oplus\ id_{\Sigma_1}\otimes\nabla^2\ \simeq\ \nabla.$$
This is a consequence of the fact that the natural isomorphism $i:Cl_p\otimes Cl_q\simeq Cl_n$ is an equivalence of representations of $Spin(p)\times Spin(q),$ i.e., for $g_1\in Spin(p)$ and $g_2\in Spin(q),$
$$i \circ\ \rho_1(g_1)\otimes\rho_2(g_2)= \rho(g_1,g_2) \circ i;$$
indeed, if $\xi_1\in Cl_p$ and $\xi_2\in Cl_q,$
$$i(\rho_1(g_1)\otimes\rho_2(g_2) (\xi_1\otimes\xi_2))=i(g_1\xi_1\otimes g_2\xi_2)=g_1\xi_1g_2\xi_2=g_1g_2\xi_1\xi_2=\rho(g_1,g_2)(i(\xi_1\otimes\xi_2)),$$
where the products in the third and fourth terms are products in $Cl_n$ (note that $\xi_1$ and $g_2$ commute since $\xi_1$ belongs to $Cl_p$ and $g_2$ is a product of an even number of vectors belonging to $\R^q$).
\end{rem}
As in the usual construction in spin geometry, the spin bundle $\Sigma$ is endowed with a natural action of the Clifford bundle $Cl(TM\oplus E):$ indeed, the Clifford product
\begin{eqnarray*}
Cl(\R^p\oplus\R^q)\times Cl_n&\rightarrow& Cl_n\\
(\eta,\xi)&\mapsto& \eta\cdot\xi
\end{eqnarray*}
is $Spin(p)\times Spin(q)$ equivariant, if the action of $Spin(p)\times Spin(q)$ on $Cl(\R^p\oplus\R^q)$ is the adjoint action, and the action on $Cl_n$ is the left multiplication: we obviously have, for $(g_1,g_2)\in Spin(p)\times Spin(q)$ and $g=g_1g_2\in Spin(n),$
$$(g\xi g^{-1})\cdot(g\eta)=g\cdot(\xi\eta)$$
for $\xi,\eta\in Cl_n.$ 

\subsection{A $Cl_n$-valued bilinear map on $\Sigma$}
Let us denote by $\tau:Cl_n\rightarrow Cl_n$ the anti-automorphism of $Cl_n$ such that
$$\tau(x_1\cdot x_2\cdots x_k)=x_k\cdots x_2\cdot x_1$$ 
for all $x_1,x_2,\ldots, x_k\in\R^n,$ where `$\cdot$'  denotes as usual the Clifford multiplication, and set
\begin{eqnarray}
\langle\langle.,.\rangle\rangle:\hspace{.5cm}Cl_n\times Cl_n&\rightarrow& Cl_n\label{def brackets 1}\\
(\xi,\xi')&\mapsto& \tau(\xi')\xi.\nonumber
\end{eqnarray}
This map is $Spin(n)-$invariant: for all $\xi,\xi'\in Cl_n$ and $g\in Spin(n)$ we have
$$\langle\langle g\xi,g\xi'\rangle\rangle=\tau(g\xi')g\xi= \tau(\xi')\tau(g)g\xi= \tau(\xi')\xi=\langle\langle \xi,\xi'\rangle\rangle,$$
since $Spin(n)\subset \{g\in Cl_n^0:\ \tau(g)g=1\};$ this map thus induces a $Cl_n-$valued map
\begin{eqnarray}
\langle\langle.,.\rangle\rangle:\hspace{.5cm}\Sigma\times \Sigma&\rightarrow& Cl_n\label{def brackets 2}\\
(\varphi,\varphi')&\mapsto& \langle\langle[\varphi],[\varphi']\rangle\rangle\nonumber
\end{eqnarray}
where $[\varphi]$ and $[\varphi']\in Cl_n$ represent $\varphi$ and $\varphi'$ in some spinorial frame $\tilde{s}\in\tilde{Q}.$
\begin{lem}
The map $\langle\langle.,.\rangle\rangle:$ $\Sigma\times \Sigma\rightarrow Cl_n$ satisfies the following properties: for all $\varphi,\psi\in\Gamma(\Sigma)$ and $X\in \Gamma(TM),$
\begin{equation}\label{scalar product property1}
\langle\langle \varphi,\psi\rangle\rangle=\tau\langle\langle\psi,\varphi\rangle\rangle
\end{equation}
and
\begin{equation}\label{scalar product property2}
\langle\langle X\cdot\varphi,\psi\rangle\rangle=\langle\langle\varphi,X\cdot \psi\rangle\rangle.
\end{equation}
\end{lem}
\begin{proof}
We have
$$\langle\langle \varphi,\psi\rangle\rangle=\tau[\psi]\ [\varphi]=\tau(\tau[\varphi]\ [\psi])=\tau\langle\langle\psi,\varphi\rangle\rangle$$
and
$$\langle\langle X\cdot\varphi,\psi\rangle\rangle=\tau[\psi]\ [X][\varphi]=\tau([X][\psi])[\varphi]=\langle\langle\varphi,X\cdot \psi\rangle\rangle$$
where $[\varphi],$ $[\psi]$ and $[X]\ \in Cl_n$ represent $\varphi,$ $\psi$ and $X$ in some given frame $\tilde{s}\in\tilde{Q}.$
\end{proof}
\begin{lem}
The connection $\nabla$ is compatible with the product $\langle\langle.,.\rangle\rangle:$
$$\partial_X\langle\langle\varphi,\varphi'\rangle\rangle=\langle\langle\nabla_X\varphi,\varphi'\rangle\rangle+\langle\langle\varphi,\nabla_X\varphi'\rangle\rangle$$
for all $\varphi,\varphi'\in\Gamma(\Sigma)$ and $X\in\Gamma(TM).$ 
\end{lem}
\begin{proof}
If $\varphi=[\tilde{s},[\varphi]]$ is a section of $\Sigma=\tilde{Q}\times_\rho Cl_n,$ we have
$$\nabla_X\varphi=\left[\tilde{s},\partial_X[\varphi]+\rho_*(\tilde{s}^*\alpha(X))([\varphi])\right],\hspace{1cm}\forall X\in\ TM,$$
where $\rho$ is the representation (\ref{rep spin p spin q}) and $\alpha$ is the connection form on $\tilde{Q};$ the term $\rho_*(\tilde{s}^*\alpha(X))$ is an endomorphism of $Cl_n$ given by the multiplication on the left by an element belonging to $\Lambda^2\R^n\subset Cl_n,$ still denoted by  $\rho_*(\tilde{s}^*\alpha(X)).$ Such an element satisfies
$$\tau\left( \rho_*(\tilde{s}^*\alpha(X))\right)=-\rho_*(\tilde{s}^*\alpha(X)),$$
and we have
\begin{eqnarray*}
\langle\langle\nabla_X\varphi,\varphi'\rangle\rangle+\langle\langle\varphi,\nabla_X\varphi'\rangle\rangle&=&\tau\{[\varphi']\}\left(\partial_X[\varphi]+\rho_*(\tilde{s}^*\alpha(X))[\varphi]\right)\\
&&+\tau\left\{\partial_X[\varphi']+\rho_*(\tilde{s}^*\alpha(X))[\varphi']\right\}[\varphi]\\
&=&\tau\{[\varphi']\}\partial_X[\varphi]+\tau\left\{\partial_X[\varphi']\right\}[\varphi]\\
&=&\partial_X\langle\langle\varphi,\varphi'\rangle\rangle.
\end{eqnarray*}
\end{proof}
\section{The spin geometry of a submanifold in $\R^n$} 
We keep the notation of the previous section, assuming moreover here that $M$ is a submanifold of $\R^n$ and that $E\rightarrow M$ is its normal bundle. Let as before $\tilde{Q}_M\stackrel{2:1}{\rightarrow} Q_{M}$ be a spin structure of $M$. Our goal is to construct $\tilde Q$ such that we obtain an identification 
$$\Sigma=\tilde{Q}\times_{\rho} Cl_n\simeq M\times Cl_n.$$
Although this type of result is used in several places in the literature, we could not find a complete statement or proof. Therefore we will give a detailed proof, which we believe may be useful in its own right.\\
Let $(e_1, \dots e_p )$ resp. $(e_{p+1},\dots e_{p+q} )$ be orthonormal frames of $TM$ resp. $E$ and $Q_{{\mathbb{R}}^n}$ the bundle of positively oriented orthonormal frames of $\mathbb{R}^n$. We can consider the map 
\begin{eqnarray*}\iota:  Q_M\times_M Q_E&\rightarrow& Q_{\mathbb{R}^n}\\
((e_1, \dots e_p ), (e_{p+1},\dots e_{p+q} ))&\mapsto& (e_1,e_2,\dots e_{p+q})
\end{eqnarray*}
given by the concatenation of frames.\\
The map
$$\tilde{Q}_M\times_M Q_E\rightarrow Q_M\times_M Q_E$$
is obviously a two-to-one covering of $Q_M\times_M Q_E$.\\
Let now $\tilde{Q}_{\mathbb{R}^n}\stackrel{2:1}{\rightarrow} Q_{{\mathbb{R}}^n}$ be the (unique) spin structure of $\mathbb{R}^n$.\\ Then the bundle $$\tilde{Q}:=(\tilde{Q}_M\times_M Q_E)\times_{Q_{\mathbb{R}^n}}\tilde{Q}_{\mathbb{R}^n}$$ is a $\spin(p)\times\spin(q)$-principal bundle over $M$ and a four-to-one covering of $Q_M\times_M Q_E$.
Observe that $\tilde Q = \tilde Q_M \times_M \tilde Q_E$, where $\tilde Q_E := \tilde Q / \spin(p)$ (and the projection $\tilde Q / \spin(q) \to \tilde Q_M$ is a map of principal $\spin(p)$-bundles, hence an isomorphism).  Moreover, $\tilde Q_E$ is a spin structure on $E$, canonically associated to the spin structures on $M$ and $\mathbb{R}^n$.
\begin{claim} Consider the bundle $$\tilde{Q}\times_{c}\spin(p+q):=(\tilde{Q}\times\spin(p+q))/(\spin(p)\times\spin(q)),$$ where $$c:Spin(p)\times Spin(q)\rightarrow Spin(p+q)$$
is the map corresponding to the isomorphism of Clifford algebras
	$Cl_p \hat{\otimes} Cl_q \cong Cl_{p+q}$.\\
Then there is a canonical isomorphism of $\spin(n)$-principal bundles, $$\tilde{Q}\times_{c}\spin(p+q)\cong\tilde{Q}_{\mathbb{R}^n}|_M.$$
\end{claim}
\begin{proof}
Consider the projection $\pi: \tilde{Q}\rightarrow \tilde{Q}_{\mathbb{R}^n}$ to the last factor. Then the map 
\begin{eqnarray*}
\tilde{\pi}: \tilde{Q}\times\spin(p+q)&\rightarrow& \tilde{Q}_{\mathbb{R}^n}\\
(\tilde{q},s)&\mapsto& s\pi(\tilde{q})
\end{eqnarray*} 
satisfies $\tilde{\pi}(s_0\tilde{q},ss_0^{-1})=\tilde{\pi}(\tilde{q},s)$ for any $s_0\in\spin(p)\times \spin(q)$, so $\tilde \pi$ descends to a map $\tilde Q \times_c \spin(p+q) \to \tilde Q_{\mathbb{R}^n}$.  The source is clearly a $\spin(p+q)$-principal bundle on $M$, as is the target, and the map is $\spin(p+q)$-equivariant and the identity over $M$.  Hence it is an isomorphism of principal bundles. 
\end{proof}
\begin{cor} If now $\rho:Spin(p)\times Spin(q)\rightarrow GL(Cl_n)$ is the map given by $\tilde{\rho}\circ c$, where $\tilde{\rho}:\spin(n)\rightarrow GL(Cl_n)$ is the representation induced by left multiplication, we get
$$\tilde{Q}\times_{\rho}Cl_n \cong \tilde{Q}_{\mathbb{R}^n}|_M\times_{\tilde{\rho}}Cl_n \cong M\times Cl_n.$$
\end{cor}
\begin{proof} The first isomorphism is immediate from the claim, and the second follows since $\tilde Q_{\mathbb{R}^n}$ is trivial.
\end{proof}

Two connections are thus defined on $\Sigma,$ the connection $\nabla$ introduced in the previous section and the trivial connection $\partial;$ they satisfy the following Gauss formula:
\begin{equation}\label{gauss formula}
\partial_X\varphi=\nabla_X\varphi+\frac{1}{2}\sum_{j=1}^p e_j\cdot B(X,e_j)\cdot\varphi
\end{equation} 
for all $\varphi\in\Gamma(\Sigma)$ and all $X\in \Gamma(TM),$ where $B:TM\times TM\rightarrow E$ is the second fundamental form of $M$ into $\R^n.$ We refer to \cite{Ba} for the proof (in a slightly different context).

\section{Spinorial representation of submanifolds in $\R^n$}
We state the main result of the paper. Let $M$ be a $p$-dimensional Riemannian manifold and $E\rightarrow M$ a bundle of rank $q,$ with a fiber metric and a compatible connection; we assume that $E$ and $TM$ are oriented and spin, with given spin structures. We keep the notation of Section \ref{sec preliminaries}.
\begin{thm}\label{thm main result}
We moreover assume that $M$ is simply connected, and suppose that $B:TM\times TM\rightarrow E$ is bilinear and symmetric. The following statements are equivalent:
\begin{enumerate}
\item There exists a section $\varphi\in\Gamma(U\Sigma)$ such that
\begin{equation}\label{killing equation}
\nabla_X\varphi=-\frac{1}{2}\sum_{j=1}^pe_j\cdot B(X,e_j)\cdot\varphi
\end{equation}
for all $X\in TM.$
\item There exists an isometric immersion $F:\ M\rightarrow \R^n$ with normal bundle $E$ and second fundamental form $B.$\\
Moreover, $F=\int\xi$ where $\xi$ is the $\R^n-$valued 1-form defined by
\begin{equation}\label{def xi}
\xi(X):=\langle\langle X\cdot\varphi,\varphi\rangle\rangle
\end{equation}
for all $X\in TM.$
\end{enumerate}
\end{thm}
The representation formula (\ref{def xi}) generalizes the classical Weierstrass representation formula; see Section \ref{section special cases}. 
\begin{rem}
Taking the trace of (\ref{killing equation}) we get
\begin{equation*}
D\varphi=\frac{p}{2}\vec{H}\cdot\varphi
\end{equation*}
where $D\varphi=\sum_{j=1}^pe_j\cdot\nabla_{e_j}\varphi,$ and $\vec{H}=\frac{1}{p}\sum_{j=1}^pB(e_j,e_j)$ is the mean curvature vector of $M$ in $\R^n.$ This Dirac equation is known to be equivalent to (\ref{killing equation}) only for $p=2$ or $3$ (see e.g. \cite{Fr,LR1,LR2,BLR}).
\end{rem}

\begin{proof} 
$2)\Rightarrow 1)$ is a direct consequence of the Gauss formula \eqref{gauss formula} for a submanifold of $\mathbb{R}^n$: 
the restriction of  parallel spinor fields of the ambient space $\mathbb{R}^n$ to the submanifold $M$ are obviously solutions of equation \eqref{killing equation} (recall that in the paper the spinors are constructed with the whole Clifford algebra). The immersion takes the form $F=\int\xi$ where $\xi$ is given by (\ref{def xi}) for the special choice $\varphi={1_{Cl_n}}_{|M},$ since, in that case, for all $X\in TM,$
$$\xi(X)=\tau[\varphi]\ [X]\ [\varphi]=[X]\simeq X$$
where $[\varphi]=\pm 1_{Cl_n}$ and $[X]\in\R^n$ represent $\varphi$ and $X$ in one of the two spinorial frames of $\R^n$ above the canonical basis.\\
$1)\Rightarrow 2)$: 
We will prove that the 1-form $\xi$ defined in \eqref{def xi} gives us indeed an immersion preserving the metric, the second fundamental form and the normal connection.  This follows directly from Propositions \ref{lem xi closed} and \ref{lem F isometry} below.
\end{proof}

\begin{prop} \label{lem xi closed}
Assume that $\varphi\in\Gamma(U\Sigma)$ is a solution of (\ref{killing equation}) and define $\xi$ by (\ref{def xi}). Then
\begin{enumerate}
\item $\xi$ takes its values in $\R^n\subset Cl_n;$
\item $\xi$ is a closed 1-form: $d\xi=0.$
\end{enumerate} 
\end{prop}
\begin{proof}
1- By the very definition of $\xi,$ we have
$$\xi(X)=\tau[\varphi][X][\varphi]$$  
for all $X\in TM,$ where $[X]$ and $[\varphi]$ represent $X$ and $\varphi$ in a given frame $\tilde{s}$ of $\tilde{Q}.$ Since $[X]$ belongs to $\R^n\subset Cl_n$ and $[\varphi]$ is an element of $Spin(n),$ $\xi(X)$ belongs to $\R^n.$
\\2- We compute, for $X,Y\in \Gamma(TM)$ such that $\nabla X=\nabla Y=0$ at some point $x_0,$
\begin{eqnarray*}
\partial_X\ \xi(Y)&=&\langle\langle Y\cdot\nabla_X\varphi,\varphi\rangle\rangle+\langle\langle Y\cdot\varphi,\nabla_X\varphi\rangle\rangle\\
&=&(id+\tau)\langle\langle Y\cdot\varphi,\nabla_X\varphi\rangle\rangle\\
&=&(id+\tau)\langle\langle  \varphi,-\frac{1}{2}\sum_{j=1}^pY\cdot e_j\cdot B(X,e_j)\cdot\varphi\rangle\rangle.
\end{eqnarray*}
Hence
\begin{eqnarray*}
d\xi(X,Y)&=&\partial_X\ \xi(Y)-\partial_Y\ \xi(X)\\
&=&(id+\tau)\langle\langle \varphi,\mathcal{C}\cdot\varphi\rangle\rangle
\end{eqnarray*}
with
$$\mathcal{C}:=-\frac{1}{2}\sum_{j=1}^p\left\{Y\cdot e_j\cdot B(X,e_j)-X\cdot e_j\cdot B(Y,e_j)\right\}.$$
Now, for $X=\sum_{1\leq k\leq p} x_k e_k$ and $Y=\sum_{1\leq k\leq p} y_k e_k,$ 
$$\sum_{j=1}^pX\cdot e_j\cdot B(Y,e_j)=-B(Y,X)+\sum_{j=1}^p\sum_{k\neq j}x_ke_k\cdot e_j\cdot B(Y,e_j)$$
and
$$\sum_{j=1}^pY\cdot e_j\cdot B(X,e_j)=-B(X,Y)+\sum_{j=1}^p\sum_{k\neq j}y_ke_k\cdot e_j\cdot B(X,e_j),$$
which yields the formula
$$\mathcal{C}=-\frac{1}{2}\sum_{j=1}^p\sum_{k\neq j}e_k\cdot e_j\cdot(y_kB(X,e_j)-x_kB(Y,e_j)).$$
This shows that $\tau[\mathcal{C}]=-[\mathcal{C}],$ which implies that
$$\tau\langle\langle \varphi,\mathcal{C}\cdot\varphi\rangle\rangle=\tau(\tau[\varphi]\tau[\mathcal{C}][\varphi])=-\tau[\varphi]\tau[\mathcal{C}][\varphi]=-\langle\langle \varphi,\mathcal{C}\cdot\varphi\rangle\rangle.$$
Thus
$$d\xi(X,Y)=(id+\tau)\langle\langle \varphi,\mathcal{C}\cdot\varphi\rangle\rangle=0.$$
\end{proof}
We keep the notation of Proposition \ref{lem xi closed}, and moreover assume that $M$ is simply connected; since $\xi$ is closed by Proposition \ref{lem xi closed}  we can consider 
$$F:M\rightarrow\R^n$$ 
such that $dF=\xi.$ The next proposition follows from the properties of the Clifford product:
\begin{prop}\label{lem F isometry}
1. The map $F:M\rightarrow \R^n$ is an isometry.
\\2. The map
\begin{eqnarray*}
\Phi_E:\hspace{1cm} E&\rightarrow& M\times\R^n\\
X\in E_m&\mapsto& (F(m),\xi(X)) 
\end{eqnarray*}
is an isometry between $E$ and the normal  bundle of $F(M)$ into $\R^n,$ preserving connections and second fundamental forms. Here, for $X\in E,$ $\xi(X)$ still stands for the quantity $\langle\langle X\cdot\varphi,\varphi\rangle\rangle.$
\end{prop}
\begin{proof}
For $X,Y\in\Gamma(TM\oplus E),$ we have
\begin{eqnarray*}
\langle \xi(X),\xi(Y)\rangle &=&-\frac{1}{2}\left(\xi(X)\xi(Y)+\xi(Y)\xi(X)\right)\\
&=&-\frac{1}{2}\left(\tau[\varphi][X][\varphi]\tau[\varphi][Y][\varphi]+\tau[\varphi][Y][\varphi]\tau[\varphi][X][\varphi]\right)\\
&=&-\frac{1}{2}\tau[\varphi]\left([X][Y]+[Y][X]\right)[\varphi]\\
&=&\langle X,Y\rangle,
\end{eqnarray*}
since $[X][Y]+[Y][X]=-2\langle[X],[Y]\rangle=-2\langle X,Y\rangle.$ This implies that $F$ is an isometry, and that $\Phi_E$ is a bundle map between $E$ and the normal bundle of $F(M)$ into $\R^n$ which preserves the metrics of the fibers. Let us denote by $B_F$ and $\nabla'^F$ the second fundamental form and the normal connection of the immersion $F;$ we want to show that
\begin{equation}\label{xi preserves ff connection}
\xi(B(X,Y))=B_F(\xi(X),\xi(Y))\hspace{.5cm}\mbox{and}\hspace{.5cm}\xi(\nabla'_XN)=\nabla'^F_{\xi(X)}\xi(N)
\end{equation}
for $X,Y\in \Gamma(TM)$ and $N\in\Gamma(E).$ First,
$$B_F(\xi(X),\xi(Y))=\left\{\partial_X\ \xi(Y)\right\}^N$$
where the superscript $N$ means that we consider the component of the vector which is normal to the immersion. We showed in the proof of Proposition \ref{lem xi closed} that fixing a point $x_0\in M,$ and assuming that $\nabla Y=0$ at $x_0$ we have 
\begin{eqnarray*}
\partial_X\ \xi(Y)=-\frac{1}{2}(id+\tau)\langle\langle \varphi,Y\cdot\sum_{j=1}^pe_j\cdot B(X,e_j)\cdot\varphi\rangle\rangle.
\end{eqnarray*}  
and that moreover 
$$Y\cdot\sum_{j=1}^pe_j\cdot B(X,e_j)=-B(X,Y)+\mathcal{D}$$
where $\mathcal{D}$ is a term which satisfies $\tau\mathcal{D}=-\mathcal{D}.$ This implies that
\begin{eqnarray*}
B_F(\xi(X),\xi(Y))&=&\left\{\frac{1}{2}(id+\tau)\langle\langle \varphi,B(X,Y)\cdot\varphi\rangle\rangle\right\}^N\\
&=&\xi(B(X,Y)),
\end{eqnarray*}
where the last equality holds since $\tau[B(X,Y)]=[B(X,Y)]$ and $\xi(B(X,Y))$ is normal to the immersion. We finally show the second identity in (\ref{xi preserves ff connection}): we have
\begin{eqnarray*}
\nabla'^F_{\xi(X)}\xi(N)&=&(\partial_X\ \xi(N))^N\\
&=&\langle\langle\nabla'_XN\cdot\varphi,\varphi\rangle\rangle^N+\langle\langle N\cdot\nabla_X\varphi,\varphi\rangle\rangle^N+\langle\langle N\cdot\varphi,\nabla_X\varphi\rangle\rangle^N.
\end{eqnarray*}
The first term in the right hand side is $\xi(\nabla'_XN),$ and we only need to show that
\begin{equation}\label{lem F isometry expr}
\langle\langle N\cdot\nabla_X\varphi,\varphi\rangle\rangle^N+\langle\langle N\cdot\varphi,\nabla_X\varphi\rangle\rangle^N=0.
\end{equation}
We have
\begin{eqnarray*}
\langle\langle N\cdot\nabla_X\varphi,\varphi\rangle\rangle+\langle\langle N\cdot\varphi,\nabla_X\varphi\rangle\rangle&=&(id+\tau)\langle\langle N\cdot\nabla_X\varphi,\varphi\rangle\rangle\\
&=&\frac{1}{2}(id+\tau)\langle\langle \sum_{j=1}^p e_j\cdot N\cdot B(X,e_j)\cdot\varphi,\varphi\rangle\rangle,
\end{eqnarray*}
and the identity (\ref{lem F isometry expr}) will thus be proved if we show that this vector is tangent to the immersion. We have 
\begin{eqnarray*}
\sum_{j=1}^p e_j\cdot N\cdot B(X,e_j)&=&-\sum_{j=1}^p e_j\cdot B(X,e_j)\cdot N-2\sum_{j=1}^p \langle B(X,e_j), N\rangle \ e_j\\
&=&-\sum_{j=1}^p B(X,e_j)\cdot N\cdot e_j-2B^*(X,N)\\
&=&-\tau \left(\sum_{j=1}^p e_j\cdot N\cdot B(X,e_j)\right)-2B^*(X,N)
\end{eqnarray*}
where we have set $B^*(X,N)=\sum_{j=1}^p \langle B(X,e_j), N\rangle \ e_j;$ thus
$$\frac{1}{2}(id+\tau)\langle\langle \sum_{j=1}^p e_j\cdot N\cdot B(X,e_j)\cdot\varphi,\varphi\rangle\rangle=-\langle\langle B^*(X,N)\cdot\varphi,\varphi\rangle\rangle,$$
which is a vector tangent to the immersion since $B^*(X,N)$ belongs to $TM;$ (\ref{lem F isometry expr}) follows, which finishes the proof. \end{proof}
\begin{rem}\label{rmk congruence}
The group $Spin(n)$ naturally acts on $U\Sigma$ by multiplication on the right, and if $\varphi\in\Gamma(U\Sigma)$ is a solution of (\ref{killing equation}) and $g_0$ belongs to $Spin(n),$ then $\varphi\cdot g_0$ is also a solution of (\ref{killing equation}); in fact, $\varphi\cdot g_0$ defines an immersion which is congruent to the immersion defined by $\varphi:$ indeed, for all $X\in\Gamma(TM),$
\begin{eqnarray*}
\xi_{\varphi\cdot g_0}(X)&=&\tau[\varphi\cdot g_0][X][\varphi\cdot g_0]\\
&=&\tau(g_0)\tau[\varphi][X][\varphi]g_0\\
&=&\tau(g_0)\xi_{\varphi}(X)g_0,
\end{eqnarray*}
i.e.
$$\xi_{\varphi.g_0}\ =Ad({g_0}^{-1})\circ\xi_{\varphi};$$
the linear part of the rigid motion between the immersions defined by $\varphi$ and $\varphi\cdot g_0$ is thus $Ad({g_0}^{-1})\in SO(n).$
\end{rem}
\section{An application: the Fundamental Theorem of Submanifold Theory}\label{section fundamental theorem}
 We first recall the equations of Gauss, Ricci and Codazzi for the symmetric bilinear form $B$. Let $R^T$ and $R^N$ stand respectively for the curvature tensors of the connections on $TM$ and on $E$. Further, let $B^*:TM\times E\rightarrow TM$ be the bilinear map such that for all $X,Y\in\Gamma(TM)$ and $N\in\Gamma(E)$
$$\langle B(X,Y),N\rangle =\langle Y,B^*(X,N)\rangle,$$
then we have, for all $X,Y,Z\in\Gamma(TM)$ and $N\in\Gamma(E),$
\begin{enumerate}
\item the Gauss equation
$$R^T(X,Y)Z=B^*(X,B(Y,Z))-B^*(Y,B(X,Z)),$$
\item the Ricci equation
$$R^N(X,Y)N=B(X,B^*(Y,N))-B(Y,B^*(X,N)),$$
\item the Codazzi equation
$$\tilde{\nabla}_X B(Y,Z)=\tilde{\nabla}_Y B(X,Z);$$
\end{enumerate}
in the last equation, $\tilde{\nabla}$ denotes the natural connection on $T^*M\otimes T^*M\otimes E.$ 
\begin{prop}\label{GRC_as_int_conditions}
The equations of Gauss, Ricci and Codazzi on $B$ are the integrability conditions of (\ref{killing equation}).
\end{prop}
\begin{proof}
We assume that $\varphi\in\Gamma(U\Sigma)$ is a solution of (\ref{killing equation}) and compute the curvature 
\begin{equation*}
R(X,Y)\varphi=\nabla_X\nabla_Y\varphi-\nabla_Y\nabla_X\varphi-\nabla_{[X,Y]}\varphi.
\end{equation*}
We fix a point $x_0\in M,$ and assume that $\nabla X=\nabla Y=0$ at $x_0.$ We have
\begin{eqnarray*}
\nabla_X\nabla_Y\varphi&=&-\frac{1}{2}\sum_{j=1}^pe_j\cdot\left(\tilde{\nabla}_X B(Y,e_j)\cdot\varphi+B(Y,e_j)\cdot\nabla_X\varphi\right)\\
&=&-\frac{1}{2}\sum_{j=1}^pe_j\cdot\tilde{\nabla}_X B(Y,e_j)\cdot\varphi-\frac{1}{4}\sum_{j,k=1}^pe_j\cdot e_k\cdot B(Y,e_j)\cdot B(X,e_k).
\end{eqnarray*}
Thus
\begin{eqnarray}
R(X,Y)\varphi&=&-\frac{1}{2}\sum_{j=1}^pe_j\cdot\left(\tilde{\nabla}_X B(Y,e_j)-\tilde{\nabla}_Y B(X,e_j)\right)\cdot\varphi\nonumber\\
&&+\frac{1}{4}\sum_{j\neq k} e_j\cdot e_k\cdot \left(B(X,e_j)\cdot B(Y,e_k)-B(Y,e_j)\cdot B(X,e_k)\right)\cdot\varphi\label{R function B}\\
&&-\frac{1}{4}\sum_{j=1}^p\left(B(X,e_j)\cdot B(Y,e_j)-B(Y,e_j)\cdot B(X,e_j)\right)\cdot\varphi.\nonumber
\end{eqnarray}
We compute the last two terms in the following lemma:
\begin{lem} \label{computation AB} Let us set
$$\mathcal{A}:=\frac{1}{4}\sum_{j\neq k} e_j\cdot e_k\cdot \left(B(X,e_j)\cdot B(Y,e_k)-B(Y,e_j)\cdot B(X,e_k)\right)$$
and
$$\mathcal{B}:=-\frac{1}{4}\sum_{j=1}^p\left(B(X,e_j)\cdot B(Y,e_j)-B(Y,e_j)\cdot B(X,e_j)\right).$$
We have
$$\mathcal{A}=\frac{1}{2}\sum_{j< k}\left\{\langle B^*(X,B(Y,e_j)),e_k\rangle- \langle B^*(Y,B(X,e_j)),e_k\rangle\right\}e_j\cdot e_k$$
and
$$\mathcal{B}=\frac{1}{2}\sum_{k<l}\left\langle B(X,B^*(Y,n_k))-B(Y,B^*(X,n_k)),n_l\right\rangle n_k\cdot n_l.$$
Here $e_1,\ldots,e_p$ and $n_1,\ldots,n_q$ are orthonormal basis of $T_{x_o}M$ and $E_{x_o},$ respectively.
\end{lem}
\begin{proof}
For the computation of $\mathcal{A},$ we notice that 
$$\sum_{j\neq k} e_j\cdot e_k\cdot B(Y,e_j)\cdot B(X,e_k)=-\sum_{j\neq k} e_j\cdot e_k\cdot B(Y,e_k)\cdot B(X,e_j),$$
and get
\begin{eqnarray*}
\mathcal{A}&=&\frac{1}{4}\sum_{j\neq k} e_j\cdot e_k\cdot \left(B(X,e_j)\cdot B(Y,e_k)+B(Y,e_k)\cdot B(X,e_j)\right)\\
&=&-\frac{1}{2}\sum_{j\neq k} \langle B(X,e_j),B(Y,e_k)\rangle e_j\cdot e_k\\
&=&-\frac{1}{2}\sum_{j< k}  \left\{\langle B(X,e_j),B(Y,e_k)\rangle-\langle B(Y,e_j),B(X,e_k)\rangle\right\}e_j\cdot e_k\\
&=&-\frac{1}{2}\sum_{j< k}\left\{\langle B^*(Y,B(X,e_j)),e_k\rangle- \langle B^*(X,B(Y,e_j)),e_k\rangle\right\}e_j\cdot e_k.
\end{eqnarray*}
For the computation of $\mathcal{B},$ we write
$$B(Y,e_j)=\sum_k\langle B(Y,e_j),n_k\rangle n_k\hspace{.5cm}\mbox{and}\hspace{.5cm}B(X,e_j)=\sum_l\langle B(X,e_j),n_l\rangle n_l$$
and get
\begin{eqnarray*}
\sum_j B(Y,e_j)\cdot B(X,e_j)&=&\sum_{kl}\sum_j\langle B(Y,e_j),n_k\rangle\langle B(X,e_j),n_l\rangle n_k\cdot n_l\\
&=&\sum_{kl}\sum_j\langle e_j,B^*(Y,n_k)\rangle\langle e_j,B^*(X,n_l)\rangle n_k\cdot n_l\\
&=&\sum_{kl}\langle B^*(Y,n_k),B^*(X,n_l)\rangle n_k\cdot n_l\\
&=&\sum_{kl}\langle B(X,B^*(Y,n_k)),n_l\rangle n_k\cdot n_l;
\end{eqnarray*}
thus
\begin{eqnarray*}
\mathcal{B}&=&\frac{1}{4}\sum_{kl}\left\langle B(X,B^*(Y,n_k))-B(Y,B^*(X,n_k)),n_l\right\rangle n_k\cdot n_l\\
&=&\frac{1}{2}\sum_{k<l}\left\langle B(X,B^*(Y,n_k))-B(Y,B^*(X,n_k)),n_l\right\rangle n_k\cdot n_l.
\end{eqnarray*}
\end{proof}
On the other hand, the curvature of the spinorial connection is given by
\begin{eqnarray}
R(X,Y)\varphi&=&\frac{1}{2}\left(\sum_{1\leq j<k\leq p}\langle R^T(X,Y)(e_j),e_k\rangle\ e_j\cdot e_k\right.\label{R function RT RN}\\
&&+\left.\sum_{1\leq k<l\leq q}\langle R^N(X,Y)(n_k),n_l\rangle\ n_k\cdot n_l\right)\cdot\varphi.\nonumber
\end{eqnarray}
We now compare the expressions (\ref{R function B}) and (\ref{R function RT RN}) using the calculations in Lemma \ref{computation AB}: since in a given frame $\tilde{s}$ belonging to $\tilde{Q},$ $\varphi$ is represented by an element which is invertible in $Cl_n$ (it is in fact represented by an element belonging to $Spin(n)$), we may identify the coefficients and get
$$\langle R^T(X,Y)(ej),e_k\rangle =\langle B^*(X,B(Y,e_j)),e_k\rangle- \langle B^*(Y,B(X,e_j)),e_k\rangle,$$
$$\langle R^N(X,Y)(n_k),n_l\rangle=\langle B(X,B^*(Y,n_k)),n_l\rangle-\langle B(Y,B^*(X,n_k)),n_l\rangle$$
and
$$\tilde{\nabla}_X B(Y,e_j)-\tilde{\nabla}_Y B(X,e_j)=0$$
for all the indices. These equations are the equations of Gauss, Ricci and Codazzi.
\\

We finally show that the equations of Gauss, Codazzi and Ricci are also sufficient to get a solution of (\ref{killing equation}): by the computation above, the connection on $\Sigma$ defined by
\begin{equation}\label{def nabla prime}
\nabla'_X\varphi:=\nabla_X\varphi+\frac{1}{2}\sum_{j=1}^pe_j\cdot B(X,e_j)\cdot\varphi
\end{equation}
for all $\varphi\in\Gamma(\Sigma)$ and $X\in\Gamma(TM)$ is then a flat connection. Moreover, this connection may be regarded as a connection on the principal bundle $U\Sigma$ (with the group $Spin(n)$ acting from the right): indeed, $\nabla$ defines such a connection (since it comes from a connection on $\tilde{Q}$ and by (\ref{def usigma})), and the right hand side term in (\ref{def nabla prime}) defines a linear map
\begin{eqnarray*}
TM&\rightarrow&\chi^{inv}_V(U\Sigma)\\
X&\mapsto&\varphi\mapsto \frac{1}{2}\sum_{j=1}^pe_j\cdot B(X,e_j)\cdot\varphi
\end{eqnarray*}
from $TM$ to the vector fields $\chi^{inv}_V(U\Sigma)$ on $U\Sigma$ which are vertical and invariant under the action of the group (these vector fields are indeed of the form $\varphi\mapsto\eta\cdot\varphi,$ $\eta\in\Lambda^2(TM\oplus E)\subset Cl(TM\oplus E)).$ Since a flat connection on a principal bundle admits a local parallel section, there exists a local section $\varphi\in \Gamma(U\Sigma)$ such that $\nabla'\varphi=0,$ and thus a solution of (\ref{killing equation}).
\end{proof}
As a consequence of Theorem \ref{thm main result} and Proposition \ref{GRC_as_int_conditions} we therefore get immediately   
\begin{cor}[Fundamental Theorem of Submanifold Theory]
We keep the hypotheses and notation of Section \ref{sec preliminaries}, and moreover assume that $M$ is simply connected and that $B:TM\times TM\rightarrow E$ is bilinear, symmetric and satisfies the equations of Gauss, Codazzi and Ricci. Then there exists an isometric immersion of $M$ into $\R^n$ with normal bundle $E$ and second fundamental form $B.$ The immersion is unique up to a rigid motion in $\R^n.$
\end{cor} 
\begin{proof}
As proved in Proposition \ref{GRC_as_int_conditions}, the equations of Gauss, Codazzi and Ricci are exactly the integrability conditions of (\ref{killing equation}). By Theorem \ref{thm main result}, with a solution $\varphi\in \Gamma(U\Sigma)$ of equation (\ref{killing equation}) at hand, $F=\int\xi$, where $\xi$ is the 1-form defined in (\ref{def xi}) is the immersion. Finally, a solution of (\ref{killing equation}) is unique up to the multiplication on the right by an element of $Spin(n)$ (since this is a parallel section of the $Spin(n)$ principal bundle $U\Sigma$, see the proof of Proposition \ref{GRC_as_int_conditions}); the multiplication on the right of $\varphi$ by an element of $Spin(n)$ and the adding of a constant vector in $\R^n$ in the last integration give an immersion which is congruent to the immersion defined by $\varphi$ (see Remark \ref{rmk congruence}).
\end{proof}
\section{Relation to the Gauss map}
We show here that the spinor field representing the immersion is an horizontal lift of the Gauss map.  Let us consider the Grassmannian $Gr_{p,n}\subset \Lambda^p(\mathbb{R}^n)$ of the oriented $p$-dimensional linear spaces in $\R^n$. Using the natural isomorphism of vector spaces between the exterior algebra over $\mathbb{R}^n$ and $Cl_n$, it identifies with the set
$$\mathcal{Q}_o=\{e_{1}\cdot e_2\cdots e_p\in Cl_n,\ e_i\in\R^n,\ |e_i|=1,\ e_i\perp e_j,\ i,j=1,\ldots,p,\ i\neq j\}.$$
We recall that for an oriented $p$-dimensional submanifold $F:M\rightarrow\mathbb{R}^n$ the Gauss map is defined as the map which assigns each point $x\in M$ to the oriented tangent space $dF(T_xM)$ considered as a vector subspace of $\mathbb{R}^n$. It can hence be seen as the map into the Grassmannian
\begin{eqnarray*}
G:\hspace{1cm} M&\rightarrow& \mathcal{Q}_o\\
x&\mapsto& dF(e_{1})\cdot dF(e_2)\cdots dF(e_p),
\end{eqnarray*} 
where $e_{1},e_2,\ldots,e_p$ is a positively oriented orthonormal basis of $T_xM.$ 

We assume that the immersion $F$ of $M$ into $\R^n$ is given by a spinor field $\varphi,$ as in Theorem \ref{thm main result}.
\begin{prop}
The spinor field $\varphi$, which is a section of $U\Sigma$, is a lift of the Gauss map: the diagram 
$$\xymatrix{
  &U\Sigma \ar[d]_{\chi}\\
   M\ar@/^/[ur]^\varphi\ar[r]_{\overline{G}} & M\times\mathcal{Q}_o
  }$$
commutes, where $\overline{G}(x)=(x,G(x))$ and the projection $U\Sigma\rightarrow M\times \mathcal{Q}_o$ is given by 
\begin{equation}\label{def p}
\chi:\hspace{.5cm} \varphi\in U\Sigma_x\hspace{.3cm}\mapsto\hspace{.3cm} \left(x,\langle\langle\omega\cdot\varphi,\varphi\rangle\rangle\right),
\end{equation}
where $\omega$ is the volume form in $Cl(T_xM)$ (the product of the elements of a positively oriented orthonormal basis of $T_xM$).

It is moreover parallel with respect to the connection 
$$\nabla'_X\varphi:=\nabla_X\varphi+\frac{1}{2}\sum_{j=1}^pe_j\cdot B(X,e_j)\cdot\varphi$$
on $U\Sigma$.
\end{prop}
\begin{proof}
We first explain why the map $\chi$ as defined indeed has target $M \times \mathcal{Q}_o$.
Consider the map
\begin{eqnarray*}
\Xi:\Sigma \times_M Cl(TM) &\rightarrow& M\times Cl_n\\
                            (\psi, c) &\mapsto& \langle\langle c\cdot\psi, \psi\rangle\rangle=:\Xi_{\psi}(c).
\end{eqnarray*}
Suppose $\psi \in U \Sigma$ and $c = e_1 \cdots e_k$ for $k$ orthonormal vectors $e_1, \ldots, e_k \in T_x M$. Then, we can rewrite $\Xi_{\psi}(c) = \langle\langle c\cdot\psi, \psi\rangle\rangle$ in any spinorial frame at $x$ as
\begin{equation}\label{gaussform}
\tau [\psi] [e_1] \cdots [e_k] [\psi] = (\tau [\psi] [e_1] [\psi]) (\tau [\psi] [e_2] [\psi]) \cdots (\tau [\psi] [e_k] [\psi]).
\end{equation}
The $k$ vectors on the right-hand side are still orthonormal, so $\Xi_{\psi}(c)$ lies in the corresponding Grassmannian $Gr_{k,n}$. Consequently $\chi(\psi) = \Xi_\psi(\omega)$ lies in $M\times \mathcal{Q}_o$.

We next verify the formula for the Gauss map. Recall that the immersion is given by $F=\int\xi,$ where $\xi$ is the 1-form defined by $\xi(X)=\langle\langle X\cdot\varphi,\varphi\rangle\rangle$ for all $X\in TM.$ Thus, $dF=\xi$. We fix a positively oriented and orthonormal frame $(e_1,\ldots,e_p)$ of $TM,$ and a spinorial frame $\tilde{s}\in\tilde{Q}$ which is above $(e_1,\ldots,e_p).$ 
Then, $\omega = e_1 \cdots e_p$.  In any spinorial frame, $\tau [\varphi] [v] [\varphi] = \xi(v)$ for all $v \in T_x M$. Therefore
\eqref{gaussform} yields that
$\chi(\varphi) = \xi(e_1) \xi(e_2) \cdots \xi(e_p) = G(x)$. This proves the first part of the proposition.

Finally, $\varphi$ is horizontal with respect to the connection $\nabla'$ since it is a solution of (\ref{killing equation}).
\end{proof}

\section{Special cases: minimal surfaces, hypersurfaces, and surfaces in $\R^4$}\label{section special cases}
\subsection{Minimal surfaces in $\R^n$}
If $J$ denotes the natural complex structure on $M,$ the 1-form
$$\tilde{\xi}(X):=\xi(X)-i\xi(JX),\hspace{1cm}X\in TM,$$
is $\C$-linear, with values in the complexified Clifford algebra $\tilde{Cl_n}=Cl_n\oplus i Cl_n;$ in general
$$F=\int\Re e\ \tilde{\xi}=\int\Re e\left(\tilde{f}(z)dz\right)$$
where $z$ is a complex parameter of $M$ and $\tilde{f}$ is a smooth function. Note that $\tilde{\xi}$ and $\tilde{f}$ take in fact their values in $\C^n:=\R^n\oplus i\R^n\subset\tilde{Cl_n}.$
\begin{prop}
The form $\tilde{\xi}$ is closed (and thus holomorphic) if and only if $\vec{H}=0.$ In that case, we have
$$F=\Re e\int\tilde{f}(z)dz$$
where $\tilde{f}$ is a holomorphic function.
\end{prop}
\begin{proof}
We assume that $(e_1,e_2)$ is a local orthonormal frame on $M,$ positively oriented, such that $\nabla e_1=\nabla e_2=0$ at a point $x_0.$ We thus have
$$d\tilde\xi(e_1,e_2)=\partial_{e_1}\left(\xi(e_2)+i\xi(e_1)\right)-\partial_{e_2}\left(\xi(e_1)-i\xi(e_2)\right).$$
Noticing that, for $j,k\in\{1,2\},$ 
\begin{eqnarray*}
\partial_{e_j}\left(\xi(e_k)\right)&=&\partial_{e_j}\langle\langle e_k\cdot\varphi,\varphi\rangle\rangle\\
&=&\langle\langle e_k\cdot\nabla_{e_j}\varphi,\varphi\rangle\rangle+\langle\langle e_k\cdot\varphi,\nabla_{e_j}\varphi\rangle\rangle\\
&=&(id+\tau)\langle\langle e_k\cdot\nabla_{e_j}\varphi,\varphi\rangle\rangle,
\end{eqnarray*}
we obtain
\begin{eqnarray*}
d\tilde\xi(e_1,e_2)&=&i\ (id+\tau)\langle\langle e_1\cdot\nabla_{e_1}\varphi+e_2\cdot\nabla_{e_2}\varphi,\varphi\rangle\rangle\\&&+(id+\tau)\langle\langle e_2\cdot\nabla_{e_1}\varphi-e_1\cdot\nabla_{e_2}\varphi,\varphi\rangle\rangle.
\end{eqnarray*}
The first term on the right hand side is
\begin{eqnarray*}
i\ (id+\tau)\langle\langle e_1\cdot\nabla_{e_1}\varphi+e_2\cdot\nabla_{e_2}\varphi,\varphi\rangle\rangle&=&i\ (id+\tau)\langle\langle \vec H\cdot\varphi,\varphi\rangle\rangle\\
&=&2i\ \langle\langle \vec H\cdot\varphi,\varphi\rangle\rangle
\end{eqnarray*}
since, by (\ref{killing equation}), 
$$D\varphi:=e_1\cdot\nabla_{e_1}\varphi+e_2\cdot\nabla_{e_2}\varphi=\vec H\cdot\varphi$$
and $\tau[\vec H]=[\vec H].$ The second term is
\begin{eqnarray*}
(id+\tau)\langle\langle e_2\cdot\nabla_{e_1}\varphi-e_1\cdot\nabla_{e_2}\varphi,\varphi\rangle\rangle&=&-(id+\tau)\langle\langle e_1\cdot\nabla_{e_1}\varphi+e_2\cdot\nabla_{e_2}\varphi,e_1\cdot e_2\cdot\varphi\rangle\rangle\\
&=&-(id+\tau)\langle\langle \vec H\cdot \varphi,e_1\cdot e_2\cdot\varphi\rangle\rangle\\
&=&0,
\end{eqnarray*}
using again that $D\varphi=\vec H\cdot\varphi$ and since $\tau\left([\vec H][e_1][e_2]\right)=-[\vec H][e_1][e_2].$ We thus obtain the formula
$$d\tilde\xi(e_1,e_2)=2i\ \langle\langle \vec H\cdot\varphi,\varphi\rangle\rangle$$
which may be written in the form
\begin{equation}\label{dxi tilde}
d\tilde\xi=-\mu^2\ \langle\langle\vec H\cdot\varphi,\varphi\rangle\rangle\ dz\wedge d\overline{z}
\end{equation}
where $\mu$ is such that the metric is $\mu^2 dzd\overline{z}.$ This gives the first part of the lemma. Assuming that $\vec H=0,$ the 1-form $\tilde \xi$ is closed, and the $\C^n-$valued function $\tilde f$ such that $\tilde\xi=\tilde f dz$ is holomorphic; the result follows.
\end{proof}
The aim now is to obtain explicit formulas in terms of holomorphic functions involving the components of the spinor field. We first note the following expression of $\tilde{f}$ in terms of the spinor field $\varphi:$ 
\begin{lem}\label{lem expr f tilde}
We have
$$\tilde{f}=\mu\left\{\tau[\varphi]e_{1}^o[\varphi]-i\tau[\varphi]e_{2}^o[\varphi]\right\}$$
where the real function $\mu$ is such that the metric is 
$$\mu^2(dx^2+dy^2)$$
in $z=x+iy,$ $[\varphi]$ represents the spinor field $\varphi$ in a spinorial frame above $(\frac{1}{\mu}\partial_x,\frac{1}{\mu}\partial_y),$ and $e_1^o,e_2^o$ are the first two vectors of the canonical basis of $\R^n\subset Cl_n.$
\end{lem}
\begin{proof}
We have
$$\tilde{f}=\tilde{\xi}\left(\partial_x\right)=\tau[\varphi][\partial_x][\varphi]-i\tau[\varphi][\partial_y][\varphi]$$
and the result follows since $[\frac{1}{\mu}\partial_x]=e_{1}^o$ and $[\frac{1}{\mu}\partial_y]=e_2^o$ in such a spinorial frame. 
\end{proof}
\subsubsection{Minimal surfaces in $\R^3$} Assuming that $n=3$ and $H=0,$ we easily get by a computation using Lemma \ref{lem expr f tilde} that
$$F=\int\Re e(\tilde{f}(z)dz)=\Re e(\int \tilde{f}(z)dz)$$
where $ \tilde{f}=\left(i\frac{f}{2}(1+g^2),\frac{f}{2}(1-g^2),fg\right)$ with 
$$f=2\mu z_1^2,\hspace{1cm}g=-i\frac{\overline{z_2}}{z_1};$$  
the complex functions $z_1,$ $z_2$ are the components of $\varphi$ in a spinorial frame above $(\frac{1}{\mu}\partial_x,\frac{1}{\mu}\partial_y),$ and the functions $f$ and $g$ are holomorphic, since so are $\sqrt{\mu}z_1$ and  $\sqrt{\mu}\overline{z_2}$ (this is a consequence of the Dirac equation $D\varphi=0,$ in $z=x+iy$). This is the classical Weierstrass representation of minimal surfaces in $\R^3.$
\subsubsection{Minimal surfaces in $\R^4$} In the case of a surface in $\R^4,$ we may also recover the explicit formulas of Konopelchenko and Landolfi \cite{KL} expressing a general immersion in terms of 4 complex functions, which are solutions of first order PDE's; the functions are holomorphic if $\vec{H}=0$. We do not include the calculations, since the general representation in Theorem \ref{thm main result} easily reduces to the spinor representation given in \cite{BLR} if $p=2$ and $n=4$ (see Section \ref{sec surface R4}), and the equivalence of this representation with the Konopelchenko-Landolfi representation is proved in \cite{RR}.
\begin{rem}
For surfaces in $\R^n,$ $n\geq 5,$ it is still possible to obtain an explicit representation in terms of the components of the spinor field which represents the surface, with holomorphic datas if $\vec H=0,$ if the bundle $E$ is assumed to be flat. We do not know if such a representation is possible without this additional assumption.   
\end{rem}

\subsection{Hypersurfaces in $\R^n$}
We set $p=n-1,$ and assume that $M$ is a $p$-dimensional Riemannian manifold and $E$ is the trivial line bundle on $M$, oriented by a unit section $\nu\in \Gamma(E).$ We moreover suppose that $M$ is simply connected and that $h:TM\times TM\rightarrow \R$ is a given symmetric bilinear form. According to Theorem \ref{thm main result}, an isometric immersion of $M$ into $\R^{p+1}$ with normal bundle $E$ and second fundamental form $B=h\nu$ is equivalent to a section $\varphi$ of $\Gamma(U\Sigma)$ solution of the Killing equation (\ref{killing equation}). Note that $Q_E\simeq M$ and the double covering
$$\tilde{Q}_E\rightarrow Q_{E}$$
is trivial, since $M$ is assumed to be simply connected. Fixing a section $\tilde{s}_E$ of $\tilde{Q}_E$ we get an injective map
\begin{eqnarray*}
\tilde{Q}_M&\rightarrow &\tilde{Q}_M\times_M\tilde{Q}_E=:\tilde{Q}\\
\tilde{s}_M&\mapsto&(\tilde{s}_M,\tilde{s}_E).
\end{eqnarray*}
Using
$$Cl_p\simeq Cl^0_{p+1}\subset Cl_{p+1}$$
(induced by the Clifford map $\R^p\rightarrow Cl_{p+1},$ $X\mapsto X\cdot e_{p+1})$,  we deduce a bundle isomorphism
\begin{eqnarray}
\tilde{Q}_M\times_{\rho} Cl_p&\rightarrow& \tilde{Q}\times_{\rho} Cl^0_{p+1}\hspace{.3cm} \subset\Sigma\label{identif spineurs}\\
\psi&\mapsto&\psi^*.\nonumber
\end{eqnarray}
It satisfies the following properties: for all $X\in TM$ and $\psi\in \tilde{Q}_M\times_{\rho} Cl_p,$
$$(X\cdot_M\psi)^*=X\cdot\nu\cdot\psi^*$$
and
$$\nabla_X(\psi^*)=(\nabla_X\psi)^*.$$
The section $\varphi\in \Gamma(U\Sigma)$ solution of (\ref{killing equation}) thus identifies to a section $\psi$ of $\tilde{Q}_M\times_{\rho} Cl_p$ solution of
\begin{equation*}
\nabla_X\psi=-\frac{1}{2}\sum_{j=1}^ph(X,e_j)e_j\cdot_M\psi=-\frac{1}{2}T(X)\cdot_M\psi
\end{equation*}
for all $X\in TM,$ where $T:TM\rightarrow TM$ is the symmetric operator associated to $h.$ We deduce the following result:
\begin{thm}\label{thm hypersurfaces}
Let $T:TM\rightarrow TM$ be a symmetric operator. The following two statements are equivalent:
\begin{enumerate}
\item there exists an isometric immersion of $M$ into $\R^{p+1}$ with shape operator $T;$
\item there exists a normalized spinor field $\psi\in \Gamma(\tilde Q_M\times_{\rho} Cl_p)$ solution of \begin{equation}\label{equation psi}
\nabla_X\psi=-\frac{1}{2}T(X)\cdot_M\psi
\end{equation}
for all $X\in TM.$
\end{enumerate}
Here, a spinor field $\psi\in \Gamma(\tilde Q_M\times_{\rho} Cl_p)$ is said to be normalized if it is represented in some frame $\tilde{s}\in\tilde{Q}_M$ by an element $[\psi]\in Cl_p\simeq Cl_{p+1}^0$ belonging to $Spin(p+1).$
\end{thm}
We will see below explicit representation formulas in the cases of the dimensions 3 and 4.
\subsubsection{Surfaces in $\R^3$} Since $Cl_2\simeq\Sigma_2$ we have
$$\tilde{Q}_M\times_{\rho} Cl_2\simeq\Sigma M,$$
and $\varphi$ is equivalent to a normalized spinor field $\psi\in\Gamma(\Sigma M)$ solution of 
$$\nabla_X\psi=-\frac{1}{2}T(X)\cdot_M\psi$$
for all $X\in TM;$ this equation is also equivalent to the equation $D\psi=H\psi.$ This is the result obtained by Friedrich in \cite{Fr}. 

We now write the representation formula (\ref{def xi}) using a special model for $Cl_3,$ and indicate how to recover the representation formula given in \cite{Fr}. We first consider the Clifford map
$$(x_1,x_2,x_3)\in\R^3\mapsto \left(\begin{array}{cc}x&0\\0&-x\end{array}\right)\in \HH(2),$$
where $x=-ix_3+j(x_1+ix_2),$ which identifies $Cl_3$ to the set
$$\left\{\left(\begin{array}{cc}x&0\\0&y\end{array}\right),\ x,y\in\HH\right\}$$
and $\R^3\subset Cl_3$ to the set of the imaginary quaternions; we also consider the ideal of $Cl_3$
\begin{equation}\label{def sigma3}
\Sigma_3=\left\{\left(\begin{array}{cc}y&0\\0&0\end{array}\right),\ y\in\HH\right\}\ \subset Cl_3,
\end{equation}
which is a model of the spin representation. Now $\varphi,$ section of $U\Sigma=\tilde{Q}\times_{\rho} Spin(3),$ is equivalent to a unit spinor field $\varphi'\in\Gamma(\tilde{Q}\times_{\rho}\Sigma_3)$ (obtained by projection) and a direct computation yields
\begin{equation}\label{rep dim 3}
\langle\langle X\cdot\varphi,\varphi\rangle\rangle=i\ \mathcal{I}m\ \langle X\cdot\varphi',\varphi'\rangle+j\langle X\cdot\varphi',\alpha(\varphi')\rangle
\end{equation}
for all $X\in TM,$ where the brackets $\langle.,.\rangle$ stand for the natural hermitian product on $\Sigma_3$ and $\alpha:\Sigma_3\rightarrow\Sigma_3$ is the natural quaternionic structure. The representation formula given by the right hand side term of (\ref{rep dim 3}) appears in \cite{Fr}. Finally, the identification (\ref{identif spineurs}) for the dimension $p=2$
\begin{eqnarray*}
\tilde{Q}_M\times_{\rho} Cl_2&\rightarrow& \tilde{Q}\times_{\rho} Cl^0_{3}\hspace{.3cm} \subset\Sigma\\
\psi&\mapsto&\psi^*
\end{eqnarray*}
identifies $\varphi\in\Gamma(U\Sigma)$ to a unit spinor field $\psi\in\Gamma(\Sigma M),$ and it may be proved by a computation that
$$\langle\langle X\cdot\varphi,\varphi\rangle\rangle=i2\mathcal{R}e\langle X\cdot\psi^+,\psi^-\rangle+j\left(\langle X\cdot\psi^+,\alpha(\psi^+)\rangle-\langle X\cdot\psi^-,\alpha(\psi^-)\rangle\right)$$
where the brackets $\langle.,.\rangle$ stand here for the natural hermitian product on $\Sigma_2$ and $\alpha:\Sigma_2\rightarrow\Sigma_2$ is the natural quaternionic structure; this is the explicit formula of the immersion in terms of $\psi$ given in \cite{Fr}.

\subsubsection{Hypersurfaces in $\R^4$} Since $Cl_3\simeq\Sigma_3\oplus\Sigma'_3$ where $\Sigma_3$ and $\Sigma'_3$ are the two (non-equivalent) irreducible representations of $Cl_3,$ we get two unit spinor fields $\psi_1\in \Gamma(\Sigma M),$ $\psi_2\in \Gamma(\Sigma' M)$ solutions of (\ref{equation psi}). Noting finally that there is a natural identification
$$i:\Sigma' M\rightarrow\Sigma M$$
satisfying 
$$i(X\cdot \psi)=-X\cdot i(\psi)$$
for all $X\in TM$ and $\psi\in\Sigma'M,$ the spinor fields $\psi_1$ and $i(\psi_2)\in\Gamma(\Sigma M)$ satisfy  
\begin{equation}\label{equations hyp R4}
\nabla_X\psi_1=-\frac{1}{2}T(X)\cdot_M\psi_1\hspace{.5cm} \mbox{and}\hspace{.5cm} \nabla_Xi(\psi_2)=\frac{1}{2}T(X)\cdot_Mi(\psi_2).
\end{equation}
We thus recover a result of \cite{LR1}: the immersion is equivalent to two spinor fields on the hypersurface which are solutions of (\ref{equations hyp R4}). We may also obtain a new explicit representation formula. On one hand, we note that
\begin{equation}\label{ident1 hyp R4}
\langle\langle X\cdot\varphi,\varphi\rangle\rangle=\left(\begin{array}{cc}0&\overline{\xi_1}x\xi_2\\ \overline{\xi_2}x\xi_1&0\end{array}\right)
\end{equation}
in $Cl^0_4,$ where $\varphi\in\Gamma(U\Sigma)$ and $X\in TM$ are respectively represented in $Cl_4^0$ by 
$$\left(\begin{array}{cc}\xi_1&0\\0&\xi_2\end{array}\right)\hspace{.5cm}\mbox{and}\hspace{.5cm}\left(\begin{array}{cc}0&x\\x&0\end{array}\right),$$
with $\xi_1,\xi_2\in\HH$ and $x\in\Im m\ \HH.$ On the other hand, $\Sigma_3$ naturally identifies to $\HH$ (see (\ref{def sigma3})) and the bilinear map 
\begin{eqnarray*}
\Sigma_3\times \Sigma_3&\rightarrow&\HH\\
(\xi,\xi')&\mapsto&\overline{\xi'}\xi
\end{eqnarray*}
induces a pairing
$$\langle\langle.,.\rangle\rangle_{\Sigma M}:\hspace{1cm}\Sigma M\times\Sigma M\rightarrow\HH$$ 
on $\Sigma M=\tilde{Q}_M\times_{\rho}\Sigma_3.$ If
$$\psi=\psi_1+\psi_2,\hspace{1cm} \psi_1\in\Sigma M,\ \psi_2\in \Sigma'M$$
is such that $\varphi=\psi^*$ (by (\ref{identif spineurs}), with $p=3$), the spinor fields $\psi_1$ and $i(\psi_2)\in \Sigma M$ are respectively represented by $\xi_1$ and $\xi_2,$ and we readily get
\begin{equation}\label{ident2 hyp R4}
\langle\langle X\cdot_M i(\psi_2),\psi_1\rangle\rangle_{\Sigma M}=\overline{\xi_1}x\xi_2.
\end{equation}
The identities (\ref{ident1 hyp R4}) and (\ref{ident2 hyp R4}) identify
$$\langle\langle X\cdot\varphi,\varphi\rangle\rangle\simeq \langle\langle X\cdot_M i(\psi_2),\psi_1\rangle\rangle_{\Sigma M};$$
this gives an explicit representation of the immersion into $\R^4$ in terms of the two spinor fields $\psi_1$ and $i(\psi_2)$ of $\Sigma M$ introduced in \cite{LR1}.
\subsection{Surfaces in $\R^4$}\label{sec surface R4}
For a surface in $\R^4$, Theorem \ref{thm main result} with $p=2$ and $n=4$ reduces to the result obtained in \cite{BLR}, since the bundle $\Sigma$ naturally identifies to the bundle $\Sigma M\otimes\Sigma E$ in that case (see Remark \ref{id sigma tensor product}, observing that the representation of $Spin(2)$ on $Cl_2$ by left multiplication is also the usual complex spin representation $\Sigma_2$). Note that we may similarly recover the main results in \cite{Bay,BP} concerning immersions in $\R^{3,1}$ and $\R^{2,2},$ if we consider in our constructions the Clifford algebras $Cl_{3,1}$ and $Cl_{2,2}$ instead of $Cl_4.$ 
\section{Spinorial representation of submanifolds in $S^n$ and $\HH^n$}
We extend here Theorem \ref{thm main result} to the other space forms.
\subsection{Submanifolds of $S^n$}
Let $M$ be a Riemannian manifold of dimension $p,$ and $E$ be a bundle on $M$ of rank $q=n-p,$ with a fibre metric and a compatible connection; we assume that $TM$ and $E$ are spin, and consider
$$\Sigma:=\tilde{Q}\times_{\rho} Cl_{n+1}$$
where $\tilde{Q}=\tilde{Q}_M\times_M \tilde{Q}_E$ is the $Spin(p)\times Spin(q)$ principal bundle given by the two spin structures and $\rho:Spin(p)\times Spin(q)\rightarrow Aut(Cl_{n+1})$ is the representation obtained by the composition of the maps 
\begin{equation}\label{spin_pq_in_spin_n+1}
Spin(p)\times Spin(q)\rightarrow Spin(n)\subset Spin(n+1)
\end{equation}
and 
\begin{equation}\label{spin_n+1 in Cl_n+1}
Spin(n+1)\rightarrow Aut(Cl_{n+1}).
\end{equation}
The maps in (\ref{spin_pq_in_spin_n+1}) correspond to the decompositions
$$\R^p\oplus\R^q=:\R^n\hspace{.5cm}\subset \R^n\oplus\R e_{n+1}=:\R^{n+1},$$ and in (\ref{spin_n+1 in Cl_n+1}) the action of $Spin(n+1)$ on $Cl_{n+1}$ is the multiplication on the left. We also define
$$U\Sigma=\tilde{Q}\times_{\rho} Spin(n+1)\ \subset \Sigma.$$
Let us denote by $\nu$ the element of the Clifford bundle $\tilde{Q}\times_{Ad} Cl_{n+1}$ such that its component in an arbitrary frame $\tilde{s}\in \tilde{Q}$ is the constant vector $e_{n+1}$ (note that for all $g\in Spin(p)\times Spin(q)\subset Spin(n)\subset Spin(n+1),$ $Ad(g)(e_{n+1})=e_{n+1}$).
\begin{thm}\label{thm rep Sn}
Let $B:TM\times TM\rightarrow E$ be a symmetric and bilinear map. The following two statements are equivalent:
\begin{enumerate}
\item There exists an isometric immersion $F$ of $M$ into $S^n$ with normal bundle $E$ and second fundamental form $B.$
\item There exists a spinor field $\varphi\in\Gamma(U\Sigma)$ satisfying
\begin{equation}\label{eqn phi Sn}
\nabla_X\varphi=-\frac{1}{2}\sum_{j=1}^pe_j\cdot B(X,e_j)\cdot\varphi+\frac{1}{2} X\cdot\nu\cdot\varphi
\end{equation}
for all $X\in TM.$
\end{enumerate}
Moreover we have the representation formula
\begin{equation}\label{formula F Sn}
F=\langle\langle\nu\cdot\varphi,\varphi\rangle\rangle\hspace{.5cm}\in S^n\subset\R^{n+1},
\end{equation}
where the brackets $\langle\langle.,.\rangle\rangle$ are defined as in (\ref{def brackets 1})-(\ref{def brackets 2}).
\end{thm}
\begin{proof}
We only prove that (2) implies (1), using the explicit formula (\ref{formula F Sn}). Setting $F=\langle\langle\nu\cdot\varphi,\varphi\rangle\rangle,$ we have 
$$F= [\varphi]^{-1} e_{n+1} [\varphi]=Ad([\varphi]^{-1})(e_{n+1})$$ 
where $[\varphi]\in Spin(n+1)$ represents $\varphi$ in some frame $\tilde{s}\in\tilde{Q}$ and $Ad:Spin(n+1)\rightarrow SO(n+1)$ is the natural double covering; thus $F$ belongs to $S^n.$ We will need the following
\begin{lem}\label{lem dF Sn}
If $\varphi\in \Gamma(U\Sigma)$ is a solution of (\ref{eqn phi Sn}) then $F=\langle\langle\nu\cdot\varphi,\varphi\rangle\rangle$ is such that, for all $X\in TM,$ 
\begin{equation}\label{dF Sn}
dF(X)=\langle\langle X\cdot\varphi,\varphi\rangle\rangle.
\end{equation}
\end{lem}
\begin{proof}
We first observe that $\nabla\nu=0:$ if $\alpha$ is the connection form on $\tilde{Q}$ and $\tilde{s}\in \Gamma(\tilde{Q})$ is a local frame, then $\nu=[\tilde{s},e_{n+1}]$ and 
$$\nabla_X\nu=[\tilde{s},\partial_Xe_{n+1}+Ad_*(\alpha(\tilde{s}_*(X)))(e_{n+1})]=0$$ 
for all $X\in TM,$ since $e_{n+1}$ is constant and $\alpha$ takes values in $\Lambda^2\R^n\subset Cl_n$. Thus, for all $X\in TM,$
\begin{eqnarray*}
dF(X)&=&\langle\langle \nu\cdot\nabla_X\varphi,\varphi\rangle\rangle+\langle\langle \nu\cdot\varphi,\nabla_X\varphi\rangle\rangle\\
&=&(id+\tau)\langle\langle \nu\cdot\nabla_X\varphi,\varphi\rangle\rangle\\
&=&-\frac{1}{2}(id+\tau)\sum_{j=1}^p\langle\langle \nu\cdot e_j\cdot B(X,e_j)\cdot\varphi,\varphi\rangle\rangle+\frac{1}{2}(id+\tau)\langle\langle X\cdot\varphi,\varphi\rangle\rangle.
\end{eqnarray*}
But
\begin{eqnarray*}
\tau\langle\langle \nu\cdot e_j\cdot B(X,e_j)\cdot\varphi,\varphi\rangle\rangle&=&\langle\langle \varphi,\nu\cdot e_j\cdot B(X,e_j)\cdot\varphi\rangle\rangle\\
&=&\langle\langle B(X,e_j)\cdot e_j\cdot \nu\cdot\varphi,\varphi\rangle\rangle\\
&=&-\langle\langle \nu\cdot e_j\cdot B(X,e_j)\cdot\varphi,\varphi\rangle\rangle
\end{eqnarray*}
since the three vectors $B(X,e_j),$ $e_j$ and  $\nu$ are mutually orthogonal, and
$$\tau\langle\langle X\cdot\varphi,\varphi\rangle\rangle=\langle\langle \varphi, X\cdot\varphi\rangle\rangle=\langle\langle X\cdot\varphi,\varphi\rangle\rangle.$$
Thus (\ref{dF Sn}) follows.
\end{proof}
By the lemma and the properties of the Clifford product, $F$ is an isometric immersion, and the map 
\begin{eqnarray*}
E&\rightarrow& TS^n \\
X\in E_m&\mapsto& (F(m),\langle\langle X\cdot\varphi,\varphi\rangle\rangle)
\end{eqnarray*} 
identifies $E$ with the normal bundle of $F(M)$ into $S^n;$ it moreover identifies the connection on $E$ with the normal connection of $F(M)$ in $S^n$ and $B$ with the second fundamental form. We omit the proof since it is very similar to the proof of Lemma \ref{lem F isometry}. 
\end{proof}
\begin{rem}
Taking the trace of (\ref{eqn phi Sn}) we get
\begin{equation}\label{dirac equation Sn}
D\varphi=\frac{p}{2}\left(\vec{H}-\nu\right)\cdot\varphi
\end{equation}
where $\vec{H}=\frac{1}{p}\sum_{j=1}^pB(e_j,e_j)$ is the mean curvature vector of $M$ in $S^n.$ 
\end{rem}
\begin{rem}
We may also obtain a proof using spinors of the fundamental theorem of submanifold theory in $S^n$, showing, as in Section \ref{section fundamental theorem}, that the equations of Gauss, Codazzi and Ricci in a space of constant sectional curvature 1 are exactly the integrability conditions of (\ref{eqn phi Sn}).
\end{rem}
We finally show how to recover the spinorial characterization of a surface in $S^3$ given by Morel in \cite{Mo}. In the model $Cl_4\simeq \mathbb{H}(2)$ we have 
$$\varphi=\left(\begin{array}{cc}[\varphi^+]&0\\0&[\varphi^-]\end{array}\right),\hspace{.5cm}F=\left(\begin{array}{cc}0&\overline{[\varphi^+]}[\nu][\varphi^-]\\-\overline{[\varphi^-]}[\nu][\varphi^+]&0\end{array}\right)$$
and
$$\xi(X)=\left(\begin{array}{cc}0&\overline{[\varphi^+]}[X][\varphi^-]\\-\overline{[\varphi^-]}[X][\varphi^+]&0\end{array}\right)$$
where $[\varphi^+],$ $[\varphi^-],$ $[\nu]$ and $[X]\in\HH$ represent $\varphi^+,$ $\varphi^-,$ $\nu$ and $X$ in some spinor frame adapted to the immersion in $S^3;$ thus Lemma \ref{lem dF Sn} gives
$$F\simeq\overline{[\varphi^+]}[\nu][\varphi^-]\hspace{.5cm}\mbox{and}\hspace{.5cm}dF(X)\simeq\overline{[\varphi^+]}[X][\varphi^-].$$
If $[\varphi^+]$ is given, this system has a solution $[\varphi^-],$ unique up to the multiplication by $S^3$ on the right. The spinor field $\varphi$ is thus essentially determined by its component $\varphi^+,$ which may be identified with a spinor field $\psi\in\Gamma(\Sigma M)$ solution of 
$$D\psi=H\psi-i\overline{\psi},\hspace{1cm}|\psi|=1;$$
details are given in \cite{BLR}. This is the spinor characterization of an immersion in $S^3$ given in \cite{Mo}.
\subsection{Submanifolds of $\mathbb{H}^n$}
We now consider the $n$-dimensional hyperbolic space $\mathbb{H}^n$ as a hypersurface of the Minkowski space $\R^{n,1}.$ Since the constructions of the paper may be also carried out in a linear space with a semi-riemannian metric, we obtain a spinor representation of a submanifold in $\mathbb{H}^n$ exactly as we did for a submanifold in $S^n.$ We thus only state the results here, and refer to the previous section for the proofs. Let $M$ be a riemannian manifold of dimension $p,$ and $E$ be a bundle on $M$ of rank $q=n-p,$ with a Riemannian fibre metric and a compatible connection; we assume that $TM$ and $E$ are spin, and consider
$$\Sigma:=\tilde{Q}\times_{\rho} Cl_{n,1}$$
where $\tilde{Q}=\tilde{Q}_M\times_M \tilde{Q}_E$ is the $Spin(p)\times Spin(q)$ principal bundle given by the two spin structures and $\rho:Spin(p)\times Spin(q)\rightarrow Aut(Cl_{n,1})$ is the representation obtained by the composition of the maps 
\begin{equation}\label{spin_pq_in_spin_n,1}
Spin(p)\times Spin(q)\rightarrow Spin(n)\subset Spin(n,1)
\end{equation}
and 
\begin{equation}\label{spin_n,1 in Cl_n,1}
Spin(n,1)\rightarrow Aut(Cl_{n,1}).
\end{equation}
The maps in (\ref{spin_pq_in_spin_n,1}) correspond to the decompositions
$$\R^p\oplus\R^q=:\R^n\hspace{.5cm}\subset \R^n\oplus\R e_{n+1}=:\R^{n,1},$$ and in (\ref{spin_n,1 in Cl_n,1}) the action of $Spin(n,1)$ on $Cl_{n,1}$ is the multiplication on the left; here $e_{n+1}$ is a vector with negative norm $-1.$ We also define
$$U\Sigma=\tilde{Q}\times_{\rho} Spin(n,1)\ \subset \Sigma.$$
Let us denote by $\nu$ the element of the Clifford bundle $\tilde{Q}\times_{Ad} Cl_{n,1}$ such that its component in an arbitrary frame $\tilde{s}\in \tilde{Q}$ is the constant vector $e_{n+1}.$
\begin{thm}\label{thm Hn}
Let $B:TM\times TM\rightarrow E$ be a symmetric and bilinear map. The following two statements are equivalent:
\begin{enumerate}
\item There exists an isometric immersion $F$ of $M$ into $\mathbb{H}^n$ with normal bundle $E$ and second fundamental form $B.$
\item There exists a spinor field $\varphi\in\Gamma(U\Sigma)$ satisfying
\begin{equation}\label{eqn phi Hn}
\nabla_X\varphi=-\frac{1}{2}\sum_{j=1}^pe_j\cdot B(X,e_j)\cdot\varphi-\frac{1}{2} X\cdot\nu\cdot\varphi\end{equation}
for all $X\in TM.$
\end{enumerate}
Moreover we have the representation formula
\begin{equation}\label{formula F Hn}
F=\langle\langle\nu\cdot\varphi,\varphi\rangle\rangle\hspace{.5cm}\in \mathbb{H}^n\subset\R^{n,1},
\end{equation}
where the brackets $\langle\langle.,.\rangle\rangle$ are defined as in (\ref{def brackets 1})-(\ref{def brackets 2}).
\end{thm}
We may also recover the spinor characterization of an immersion of a surface in $\mathbb{H}^3$ given by Morel in \cite{Mo}: if $M$ is a surface and $(e_1,e_2)$ is an orthonormal basis of $TM,$ setting $\vec H_{\mathbb H^3}:=\frac{1}{2}\left(B(e_1,e_1)+B(e_2,e_2)\right)$ we see that (\ref{eqn phi Hn}) is equivalent to
$$D\varphi=(\vec H_{\mathbb H^3}+\nu)\cdot\varphi$$
where $\varphi$ is a spinor field which is represented in a frame $\tilde{s}\in\tilde{Q}$ by $[\varphi]$ belonging to $Spin(3,1).$ This is exactly the spinor representation of an immersion in $\mathbb{H}^3$ as described in \cite{Bay} Section 5, where it is moreover proved that it is equivalent to the spinor characterization given in \cite{Mo}. 
\\
\\\textbf{Acknowledgments:} P. Bayard was supported by the project PAPIIT IA105116.

\end{document}